\documentclass[hidelinks,12pt,a4paper,reqno]{amsart}
\usepackage{graphicx}
\usepackage[all]{xy}
\usepackage{amssymb}
\usepackage{amsmath}
\usepackage[utf8]{inputenc}
\usepackage{amsthm}

\newtheorem{proposition}{Proposition}
\newtheorem{remark}{Remark}
\newtheorem{definition}{Definition}
\newtheorem{theorem}{Theorem}
\newtheorem{example}{Example}
\newtheorem{corollary}{Corollary}

\usepackage[backref=page]{hyperref}
\renewcommand*{\backref}[1]{}
\renewcommand*{\backrefalt}[4]{%
    \ifcase #1 (Not cited.)%
    \or        (Cited on page~#2.)%
    \else      (Cited on pages~#2.)%
    \fi}
\newcommand{\C}{\mathbf{C}}
\newcommand{\B}{\mathbf{B}}
\newcommand{\N}{\mathbb{N}}

\newcommand{\A}{\mathcal{A}}
\newcommand{\Grp}{\mathbf{Grp}}
\newcommand{\Mon}{\mathbf{Mon}}
\newcommand{\Ord}{\mathbf{Ord}}
\newcommand{\RNMono}{\mathbf{RNMono}}
\newcommand{\Mono}{\mathbf{Mono}}

\newcommand{\Set}{\mathbf{Set}}
\renewcommand{\S}{\mathcal{S}}
\newcommand{\OrdMon}{\mathbf{OrdMon}}

\newcommand{\End}{\mathrm{End}}

\begin{document}

\title[Schreier split extensions of preordered monoids]{Schreier split extensions of preordered monoids}


\author{Nelson Martins-Ferreira}
\address[Nelson Martins-Ferreira]{Instituto Polit\'{e}cnico de Leiria, Leiria, Portugal}
\thanks{This work is supported by the Funda\c{c}\~{a}o para a Ci\^{e}ncia e a Tecnologia (FCT) and Centro2020 through the following Projects: UIDB/04044/2020, UIDP/04044/2020, PAMI - ROTEIRO/0328/2013 (Nº 022158), Next.parts (17963); and MATIS (CENTRO-01-0145-FEDER-000014 - 3362).
 }
\email{martins.ferreira@ipleiria.pt}

\author{Manuela Sobral}
\address[Manuela Sobral]{CMUC and Departamento de
Matem\'atica, Universidade de Coimbra, 3001--501 Coimbra,
Portugal}
\thanks{
The second author was partially supported by the Centre for Mathematics of
the University of Coimbra -- UID/MAT/00324/2019}
\email{sobral@mat.uc.pt}

\vspace*{1mm}

\dedicatory{Dedicated to J. M. Esgalhado Valen\c{c}a on the occasion of his 70th birthday}

\begin{abstract}
Properties of preordered monoids are investigated and important subclasses of such structures are studied. The corresponding full subcategories of the category of preordered monoids are functorially related between them as well as with the categories of preordered sets and monoids.
Schreier split extensions are described in the full subcategory of preordered monoids whose preorder is determined by the corresponding positive cone.



\end{abstract}

\maketitle

\date{Received: date / Accepted: date}

\today.

\section{Introduction}

Preordered monoids are monoids equipped with a preorder compatible with the monoid operation. They are relevant tools in many areas as, for instance, in computer science where they are used in the theory of language recognition (see \cite{RS97}),  as well as in non-classical logics, namely in fuzzy logics (see \cite{H98} and \cite{JV19}), to mention a few.

\vspace*{1mm}

Many fundamental results had been obtained by switching from categories of monoids to categories of preordered or ordered monoids, and the same for semigroups. Examples of this fact are new proofs of two remarkable results that we refer next.

A celebrated result of I. Simon (\cite{S75}) on the classification of recognizable languages in terms of $\mathcal{J}$-triviality of the corresponding syntactic monoids has a radically new proof in \cite{ST85} where it is proved that every finite
 $\mathcal{J}$-trivial monoid (for the Green's $\mathcal{J}$-equivalence relation \cite{Green}) is a quotient of an ordered monoid satisfying the identity $x \leq 1$.
In \cite{HP99}, the authors give another proof of this result and explain its relevance in the theory of finite semigroups. A
systematic use of ordered monoids in language theory, was initiated by J.-E. Pin in \cite{Pin95} and developed in \cite{PW97}, \cite{PW02} and other subsequent papers.

The second example is a new proof of a well-known and important result of A. Tarski that gives a criterion for the existence of a monoid homomorphism from a given commutative monoid $A$ to the extended positive real line $\overline {\mathbb{ R}^+}$ that sends a fixed element $a \in A$ to the 1.
  In \cite{FW92}, F. Wehrung proves that this is an Hahn-Banach type property,
stating the injectivity of $\overline {\mathbb{R}^+}$, not in the category of commutative monoids, where there are no nontrival injectives, but in the category of commutative monoids equipped with a preorder that makes every element positive,  called there ``positively ordered monoids" or P.O.M. for short.

\vspace*{1mm}

Preordered monoids have a much richer diversity of features than preordered groups.
In contrast with the case of preordered groups, in preordered monoids the submonoid of positive elements, called the positive cone, neither determines the preorder nor  is  a cancellative monoid, in general. These features of preordered groups are rescued in the new context by considering
convenient subcategories of the category of preordered monoids, $\Ord\Mon$, satisfying these properties or appropriate generalizations, covering a wide
range of structures.

In particular, the failure of the first property gives rise to a classification of preordered monoids according to the relation between its preorder and the preorder induced by the corresponding positive cone considered here that is the opposite of  Green's preorder $\mathcal{L}$
 as explained in Section \ref{sec:2}. Furthermore, this last preorder may or may not be compatible with the monoid operation. The characterization of the positive cones inducing compatible preorders provides a reason why the commutativity of the underlying monoid is often assumed in the literature.

This classification gives rise to several categories and functors between them, some of them being part of adjoint situations.

The cancellation property is often replaced by weaker conditions like the ``pseudo-cancellation" introduced in \cite{FW92} that plays an important role in the characterization of the injective objects presented there.

\vspace*{1mm}

We prove that the forgetful functors from $\OrdMon$ to $\Mon$ and to $\Ord$ are topological and monadic functors, respectively, and derive some consequences of these facts. By $\Ord$ we mean the category of preordered sets and monotone maps.

Due to the fact that $\OrdMon$ is the category $\Mon(\Ord)$ of internal monoids in $\Ord$ (which fails to be so in $\Ord\Grp$), we show that the construction of the left adjoint to $U_1\colon{\Ord\Mon=\Mon(\Ord)\to \Ord}$ as well as its monadicity can be derived from general results for the forgetful functor $\Mon(\C)\to \C$, when $\C$ is a symmetric monoidal category satisfying some additional conditions, presented in \cite{K80}, \cite{L10} and \cite{P08}.

\vspace*{1mm}

In \cite{JV19} coextensions of commutative pomonoids (monoids equipped with a compatible partial order) are introduced, generalizing similar constructions due to P. A. Grillet (\cite{G74}) and J. Leech (\cite{L82,L75}), in the unordered case.

Schreier split extensions of monoids, that first appeared in \cite{MMS13}, correspond to an important class of split epimorphisms of monoids, the Schreier split epimorphisms (whose name was inspired by the Schreier internal categories in monoids introduced by Patchkoria in \cite{P98}). Indeed, they are exactly those split epimorphisms that correspond to monoid actions: an action of a monoid $B$ on a monoid $X$ being a monoid homomorphism $\varphi\colon{B\to\End(X)}$ from $B$ to the monoid of endomorphisms of $X$. Also this class of split epimorphisms has essentially all homological and algebraic properties of the split homomorphisms in groups (see \cite{BMMS13} and \cite{BMMS14}).

Schreier split extensions have already been defined in categories of monoids with operations (\cite{MMS13}) and in the categories of cancellative conjugation monoids (\cite{GMRS19}).

\vspace{1mm}

In this paper we describe Schreier split extensions in the full subcategory $\Ord\Mon^{*}$ of $\Ord\Mon$ with objects all preordered monoids whose preorder is induced by the corresponding positive cones.

\vspace{1mm}

 In \cite{CMM19} the structure of the split extensions in the category of preordered groups is studied and the case where the restriction to the positive cones gives a Schreier split epimorphism in $\Mon$ is analysed. Also the behaviour of the category $\Mon(\Ord)$ and, more generally, the one  $\Mon(\C)$ when $\C$ satisfies suitable conditions, is considered in the last section.

 \vspace*{1mm}

Throughout we will denote preordered monoids additively, say by $(A,+,0,\leq)$ where the monoid $(A,+,0)$ is not necessarily commutative and $\leq$ is a preorder compatible with $+$, that is, where $+\colon{A\times A\to A}$ is a monotone map.

For concepts in category theory that are not defined here we suggest MacLane's book \cite{MacLane}.

\section{The Category of preordered monoids}\label{sec:2}

We start by recalling that
if $(A,+,0,\leq)$ is a preordered group, i.e. $(A,+,0)$ is a (not necessarily abelian) group and the preorder $\leq$ is compatible with the group operation
\begin{equation*}
\forall a,b,c,d\in A\quad a\leq b \quad\text{and}\quad c\leq d \quad \Longrightarrow \quad a+c\leq b+d,
\end{equation*}
then $P=\{a\in A\mid 0\leq a\}$ is a submonoid of $A$  closed under conjugation. Furthermore, this monoid $P$, that is called the positive cone of the preordered group, determines the preorder, i.e.,
\begin{equation*}
a\leq b  \Longleftrightarrow b-a\in P.
\end{equation*}
 Indeed, if $a\leq b$, since $-a\leq -a$, then $$0=a-a\leq b-a.$$
 Conversely, if $b-a\geq 0$, since $a\geq a$ then $$b=b-a+a\geq a.$$
In this case, defining $$a\leq_{P} b \Longleftrightarrow b\in P+a=a+P$$
we have that $\leq$ coincides with $\leq_P$.

\vspace*{.5cm}

In $\Ord\Mon$, if we consider the preorder $\leq_P$ defined by
\[ a\leq_P b \quad\text{if}\quad b\in P+a,\] then we get a preorder $\leq_P$ that is contained in the original preorder.

\begin{proposition}
If $(A,+,0,\leq)\in \Ord\Mon$ then $P=\{a\in A\mid 0\leq a\}$ is a submonoid of $A$ and
\[a\leq_{P} b  \Longrightarrow a\leq b.\]
\end{proposition}
 \begin{proof}
 We have that $0\in P$ and if $a,b\in P$ then $a\geq 0$ and $b\geq 0$ implies that $a+b\geq 0$ and so $P$ is a submonoid of $A$.

 If $b=x+a$ with $x\in P$, since $x\geq 0$ and $a\geq a$, then $b=x+a\geq a$.
 \end{proof}

The converse of this result is false, in general, as the following example shows.

\begin{example}\label{eg: positive cone compatible but not determinant}
Let $(A,+,0)$ be the monoid with the following addition table
\[
\begin{tabular}{c|ccccc}
+ & 0 & 1 & 2 & 3 & 4 \\
\hline
0 & 0 & 1 & 2 & 3 & 4 \\
1 & 1 & 1 & 4 & 4 & 4 \\
2 & 2 & 2 & 4 & 4 & 4 \\
3 & 3 & 3 & 4 & 4 & 4 \\
4 & 4 & 4 & 4 & 4 & 4	
\end{tabular}
\]
equipped with the preorder $\leq$ with $P=A$ and generated by the following diagram (where the arrows from zero have been omitted)
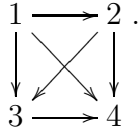
\[
\xymatrix{1\ar[r]\ar[d]\ar[rd]& 2 \ar[d]\ar[dl]\\ 3 \ar[r] & 4}.
\]
Then $(A,+,0,\leq)\in \Ord\Mon$ and $\leq_P$ is the preorder
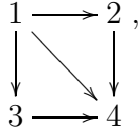
\[
\xymatrix{1\ar[r]\ar[d]\ar[rd]& 2 \ar[d]\\ 3 \ar[r] & 4},
\]
 that is strictly contained in $\leq$.
\end{example}

 In the previous example one can easily check that $\leq_P$ is compatible with $+$ and so  $(A,+,0,\leq_P)$ is also a preordered monoid.
The following example shows that this is not always the case. 

\begin{example}\label{eg: positive cone not compatible}
We consider the monoid $(A,+,0)$ with addition table
\[
\begin{tabular}{c|ccccc}
+ & 0 & 1 & 2 & 3 & 4 \\
\hline
0 & 0 & 1 & 2 & 3 & 4 \\
1 & 1 & 1 & 2 & 2 & 4 \\
2 & 2 & 1 & 2 & 1 & 4 \\
3 & 3 & 1 & 2 & 1 & 4 \\
4 & 4 & 4 & 4 & 4 & 4
\end{tabular}
\]
with $P=A$ and the preorder generated by
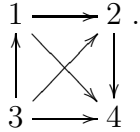
\[
\xymatrix{1\ar[r]\ar[rd]& 2 \ar[d]\\ 3\ar[u]\ar[ur] \ar[r] & 4}.
\]
It is easy to check that $(A,+,0,\leq)$ is a preordered monoid. However,  $\leq_P$ being the following preorder
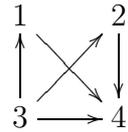
\[
\xymatrix{1\ar[rd]& 2 \ar[d]\\ 3\ar[u]\ar[ur] \ar[r] & 4}
\]
is not compatible with the monoid operation. Indeed, $2\geq_P 2$ and $1\geq_P 1$ but $1+2=2\ngeq_P 1$ since $2\notin A+1=\{1,4\}$.
\end{example}

The following is an example of a preordered monoid where the two preorders coincide.

\begin{example}\label{eg: a different positive cone}
Let $(A,+,0)$ be the monoid of Example \ref{eg: positive cone compatible but not determinant} now with a different positive cone, $P=\{0,1\}$, and the preorder sketched below
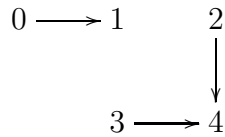
\[\xymatrix{0\ar[r]&1&2\ar[d]\\& 3\ar[r]&4}\]
which is exactly $\leq_P$, i.e. $\leq$ is the same as $\leq_P$.
\end{example}

Now we characterize the submonoids of a preordered monoid which induce a compatible preorder.

\begin{definition}\label{def: normal submonoid}
Given a monoid $A$ and a submonoid $M$ of $A$ we say that $M$ is
\begin{enumerate}
\item[-] \emph{right normal} if $a+M\subseteq M+a$, for every $a\in A$;
\item[-] \emph{left normal} if $M+a\subseteq a+M$, for every $a\in A$;
\item[-] \emph{normal} if it is both right and left normal.
\end{enumerate}
\end{definition}

\begin{proposition}\label{prop: iff P is right normal}
Let $P$ be the positive cone of a preordered monoid $(A,+,0,\leq)$. Then the monoid operation is monotone with respect to $\leq_P$ if and only if $P$ is right normal.
\end{proposition}

\begin{proof}
If $\leq_P$ is compatible with $+$ and $b=a+x$ with $x\in P$ then
\[x\geq_P 0\quad\text{and}\quad a\geq_P a\quad \Longrightarrow\quad b=a+x\geq_P a\] and so there exists an $y\in P$ such that $a+x=y+a$, i.e. $a+P\subseteq P+a$.

Conversely, if $a\leq_P b$ and $c\leq_P d$ then $b=x+a$ and $d=y+c$, for some $x,y\in P$ and so, because $P$ is right normal, we can find $z\in P$ for which $a+y=z+a$, hence
\[b+d=x+a+y+c=x+z+a+c\] and so $a+c\leq_P b+d$.
\end{proof}

In Example \ref{eg: positive cone compatible but not determinant} we have $P=A$, the so-called positively preordered monoids, and the left and right cosets are the following
\[
\begin{tabular}{c|c|c}
a & a+A & A+a \\
\hline
 0 & A & A \\
1 & \{1,4\} & \{1,2,3,4\} \\
2 & \{2,4\} & \{2,4\} \\
3 & \{3,4\} & \{3,4\} \\
 4 & \{4\} & \{4\}
\end{tabular}
\]
Since $P$ is right normal --- for all $a\in A$, $a+A\subseteq A+a$ --- then $\leq_P$ is compatible with $+$.

For Example \ref{eg: positive cone not compatible}, again $P=A$ but $P$ is not right normal and so $\leq_P$ is not compatible with $+$.
\[
\begin{tabular}{c|c|c}
a & a+A & A+a \\
\hline
 0 & A & A \\
1 & \{1,2,4\} & \{1,4\} \\
2 & \{1,2,4\} & \{2,4\} \\
3 & \{1,2,3,4\} & \{1,2,3,4\} \\
 4 & \{4\} & \{4\}
\end{tabular}
\]
We remark that, in this case, $A$ is not right normal in itself but it is left normal --- $A+a\subseteq a+A$, for every $a\in A$ --- and so if we consider the preorder
\[a\leq_{P}^{'} b \Longleftrightarrow b\in a+P\]
then, using a result similar to the one of Proposition \ref{prop: iff P is right normal}, we conclude that $(A,+,0,\leq_{P}^{'})\in \Ord\Mon$.

\begin{remark}\label{remark:1}
For a submonoid M of a monoid A we have that $\leq_M = {{\leq}^{op}}_{\mathcal {L}}$
and ${{\leq}'}_M = {\leq^{op}}_{\mathcal {R}}$, where $\mathcal {L}$ and $\mathcal {R}$ are the
Green's relations defined, in additive notation, by
$$a \leq_{\mathcal {L}} b \Leftrightarrow M + a \subseteq M + b,$$
$$a \leq_{\mathcal {R}} b \Leftrightarrow a + M \subseteq b + M.$$
Indeed,
\[ a \leq_M b \Leftrightarrow b = x + a, \mbox{ for some } x \in M  \Leftrightarrow M + b \subseteq M + a
 \Leftrightarrow b \leq_{\mathcal {L}} a,\]
 and the same for $\leq_{\mathcal {R}}$.

\end{remark}

\begin{corollary}\label{cor: 2.7} For every submonoid $M$ of a commutative preordered monoid $(A,+,0,\leq)$, the preorders $\leq_M$ and $\leq_M^{'}$ coincide and, moreover, $(A, +,0,\leq_M)$ is a preordered monoid.
\end{corollary}

Obviously, the positive cone of a commutative preordered monoid need not determine the preorder: for
\[\begin{tabular}{c|ccc|}
+ & 0 & 1 & 2 \\
\hline
0 & 0 & 1 & 2 \\
1 & 1 & 1 & 1 \\
2  & 2 & 1 & 1
\end{tabular} \]
with $P=A$  and $\leq$ as sketched below
\[\xymatrix{1\ar@<.5ex>[r] & 2 \ar@<.5ex>[l]}\]
the right (= left) cosets are
\[
\begin{tabular}{c|c}
a & P+a \\
\hline
0 & P \\
1 & \{1\} \\
2 & \{1,2\} \\
\end{tabular}
\]
and so $\leq_P$ is
\[\xymatrix{1 & 2 \ar[l] },\]
but $1\leq_P 2$ because $2\notin P+1$.

Let us denote by $\Ord\Mon^{*}$ the full subcategory of $\OrdMon$ with objects the preordered monoids such that $\leq = \leq_P$. And the same for the commutative case, $\Ord\C\Mon^{*}$.

\begin{proposition}\label{prop: 2.9} The subcategory $\Ord\C\Mon^{*}$ is coreflective in the category $\Ord\C\Mon$.
\end{proposition}
\begin{proof}
If $(A,+,0,\leq)$ is a  preordered commutative monoid and $P$ is its positive cone then, by Corollary \ref{cor: 2.7}, $(A,+,0,\leq_P)\in \Ord\C\Mon^{*}$. Furthermore, the identity morphism $c_A\colon{(A,\leq_P)\to (A,\leq)}$ is the coreflection. Indeed, given a morphism $f\colon{(A',\leq_{P'})\to (A,\leq)}$ in $\Ord\C\Mon$ if $a'\in P'$ then $f(a')\in P$ ($a'\geq 0 \Rightarrow f(a')\geq 0$) and so $f(P')\subseteq P$. Consequently $f$ factors through $c_A$
\[\xymatrix{(A,\leq_P)\ar[r]^{C_{(A,\leq)}}& (A,\leq)\\(A',\leq_{P'})\ar@{-->}[u]^{\bar{f}}\ar[ur]_{f}}\]
by a unique homomorphism $\bar{f}\in Ord\C\Mon^{*}$ because if $a'_1\leq_{P'} a'_2$ then $a'_2\in P'+a'_1$ and so
\[f(a'_2)\in f(P')+f(a'_1)\subseteq P+f(a'_1).\]
Hence, $f(a'_1)\leq_P f(a'_2)$ and so $\bar{f}(a'_1)\leq_P \bar{f}(a'_2)$ for all $a'_1\leq_{P'} a'_2$ in $A'$.
\end{proof}

\begin{definition}\label{def: right normal mono}
We say that a monomorphism of monoids $m\colon{S\to A}$ is \emph{right normal} if $m(S)$ is a right normal submonoid of $A$ and we denote by $\RNMono(\Mon)$ the corresponding  full subcategory of the category of monomorphisms of monoids, $\Mono(\Mon)$.
\end{definition}

Example \ref{eg: positive cone not compatible}  shows that the identity morphisms may not be a right normal monomorphism.

\begin{theorem}\label{thm: 2.11}
The category $\Ord\Mon^{*}$ is isomorphic to the one of right normal monomorphisms in $\Mon$, $\RNMono(\Mon)$.
\end{theorem}
\begin{proof}
The functor $G\colon{\Ord\Mon^{*}\to \RNMono(\Mon)}$ defined by
\[
\xymatrix{(A,\leq_P) \ar[d]_{f}\\(A',\leq_{P'})}\mapsto \xymatrix{ P\ar[r]^{}\ar[d]_{f|_P} & A \ar[d]^{f} \\ P'\ar[r] & A'}
\]
has an inverse $F\colon{\RNMono(Mon)\to \Ord\Mon^{*}}$ assigning
\[
\xymatrix{ S\ar[r]^{}\ar[d]_{f'} & A \ar[d]^{f} \\ S'\ar[r] & A'}\mapsto \xymatrix{(A,\leq_S) \ar[d]_{f}\\(A',\leq_{S'})}
\]
where $f(S)\subseteq S'$ implies that $f\in \Ord\Mon^{*}$. Then $GF(S\to A)=G(A,\leq_S)=(S\to A)$ and $FG(A,\leq_P)=F(P\to A)=(A,\leq_P)$.
\end{proof}

\vspace*{2mm}

The following are examples, inspired by \cite{FW92}, of  objects in $\Ord\Mon^{*}$.

\begin{enumerate}

\item  The set of all $R$-submodules of a module $A$ over a ring $R$, equipped with the ``Minkovski sum"
\[ U + V = \{ u + v : u \in U \mbox{ and } v \in V \} \] and the order defined by the inclusion. Indeed, in this case every element is positive and  $U\subseteq V$ if and only if $V=V+U$.

\item All injective objects in $\OrdMon$  with respect to embeddings (not to monomorphisms) are objects in $\Ord\Mon^{*}$. In fact, let $M$ be the submonoid of the monoid $\N \times \N$,
    generated by $(1, 0)$ and $(1, 1)$ with the order induced by the product order and $i\colon M \to \N \times \N$ the embedding. If  $a \leq b$ in an injective object $A$ then there exists a (unique) morphism in $\OrdMon$, $u \colon M \to A$
    such that $u(1, 0)= a$ and $u(1, 1)= b$, defined by $u(n+m,m)=na+mb$, for every $n,m\in \N$. By injectivity of $A$, there exists a morphism $v\colon{ \N \times \N \to A}$
\[\xymatrix{M \ar[r]^(.4){i} \ar[d]_{u} & \N\times \N \ar[ld]^{v}\\ A}\]
    extending $v$, that is such that $v \cdot i = u$. Then taking $c = v(0, 1)$ we have that $b = c + a \in P + a$ and so the preorder in $A$ coincides with the one induced by its positive cone. Indeed, since $(0,0)\leq (0,1)$ and $v$ preserves the order then $0\leq c$.

\end{enumerate}

\vspace*{2mm}

Let $\Ord\Mon^{\square}$ be the full subcategory of $\Ord\Mon$ with objects all preordered monoids whose positive cone is a right normal monoid.

\begin{proposition}\label{prop: 2.12}
The category $\Ord\Mon^{*}$ is coreflective in $\Ord\Mon^{\square}$.
\end{proposition}

\begin{proof}
Essentially the same as the one of Proposition
 \ref{prop: 2.9}.
\end{proof}

Summing up, we have the following commutative diagram of categories and functors
\[
\xymatrix{\Ord\Mon \ar[rrr] &&& \Mono(\Mon)\\
& \Ord\Mon^{\square} \ar[ul]\ar@<-.5ex>[dl] \\ \Ord\Mon^{*} \ar[uu] \ar@<-.5ex>[ur] \ar@<.5ex>[rrr]^{\cong} &&& \RNMono(\Mon) \ar@<.5ex>[lll] \ar[uu]}
\]
where $\Ord\Mon^{*}$ is coreflective in $\Ord\Mon^{\square}$ but $\Ord\Mon^{\square}$ is not coreflective in $\Ord\Mon$ as we prove in the following section.

\section{The forgetful functors}

Let us consider the following commutative diagram of forgetful functors
\[
\xymatrix{\Ord\Mon \ar[r]^{U_2} \ar[d]_{U_1} & \Mon\ar[d]^{V_1}\\ \Ord\ar[r]_{V_2} & \Set}
\]
where $V_2$ is topological and $V_1$ is a monadic functor. We are going to prove that also $U_2$ is a topological functor and $U_1$ is a monadic one.

\begin{proposition}\label{prop: 3.1}
The functor $U_2 \colon \Ord\Mon\to\Mon$ is a topological functor.
\end{proposition}
\begin{proof}
Given a family of monoid homomorphisms \[f_i\colon{(X,+,0)\to U_2(A_i,+,0,\leq_{i})},\] for $i\in I$, defining for $x,x'\in X$ \[x\leq x' \Leftrightarrow f_i(x)\leq_{i}f_i(x'),\forall i\in I,\]
we obtain a preorder which, in addition, is compatible with the monoid operation:
\begin{eqnarray*}
x\leq x'\text{ and }y\leq y' &\Leftrightarrow& \forall i\in I, f_i(x)\leq f_i(x') \text{ and } f_i(y)\leq f_i(y')\\
&\Leftrightarrow& \forall i\in I, f_i(x)+ f_i(y) \leq f_i(x')+ f_i(y')\\
&\Leftrightarrow& \forall i\in I, f_i(x+y) \leq f_i(x'+y')\\
&\Leftrightarrow& x+y\leq x'+y'.
\end{eqnarray*}
\end{proof}

From that we conclude that:
\begin{enumerate}
\item $U_2$ has a left and a right adjoint defined by equipping each monoid with the discrete and the total
preorder, respectively;
\item $\Ord\Mon$ is complete and cocomplete, since $\Mon$ is complete and cocomplete, and $U_2$ preserves limits and colimits.
\end{enumerate}

\begin{proposition}
The functor $U_1 \colon \Ord\Mon \to \Ord $ has a left adjoint.
\end{proposition}
\begin{proof} Let $F(X, \leq) = (X^{*}, \text{con}, [\;], \leq)$, where $X^{*}$ is the set of all words in the alphabet $X$ with the operation of concatenation,  having the empty word $[\;]$ as identity (the free monoid on the set $X$), equipped with the preorder
\[w =[w_1 \cdots w_n] \leq w' = [{w'}_1 \cdots {w'}_m] \] if and only if
$n = m$ and $w_i \leq {w'}_i$ for $ i = 1, 2 , \cdots n$.
This way we define a preorder compatible with concatenation.

The morphism
\[ \eta_{(X, \leq)} \colon (X, \leq) \to U_1(X^{*}, \text{con}, [\;], \leq), \] which assigns to each $x\in X$ the singular word $[ x ]$, is universal from $(X, \leq)$ to $U_1$:
\[\xymatrix{(X,\leq)\ar[r]^(.35){\eta_{(X,\leq)}}\ar[rd]_{f} & U_1(X^{*}, \text{con}, [\;], \leq)\ar@{-->}[d]^{U_1\bar{f}} & (X^{*}, \text{con}, [\;], \leq)\ar@{-->}[d]_{\bar{f}}\\ & U_1(A,+, 0, \leq) & (A,+, 0, \leq)}\]
for each $f$ in $\Ord$ there exists a unique $\bar{f}\in \Ord\Mon$ such that $\bar{f}([x])=f(x)$ and so $\bar{f}([x_1\,x_2\,\cdots\, x_n])=f(x_1)+f(x_2)+\cdots + f(x_n)$, because $\bar{f}\in \Mon$. And $\bar{f}$ is monotone: if $x=[x_1\,x_2\,\cdots\, x_n]\leq y=[y_1\,y_2\,\cdots\, y_n]$, since $x_i\leq y_i$, $i=1,\cdots,n$, then $f(x_1)+f(x_2)+\cdots + f(x_n)\leq f(y_1)+f(y_2)+\cdots + f(y_n)$, i.e. $f(x)\leq f(y)$.

Consequently, this defines a functor  $$F_1\colon{\Ord\to\Ord\Mon}$$ that is left adjoint of $U_1$ with unit $\eta$.
\end{proof}

\begin{proposition}\label{prop:3.3}
The functor $U_1\colon{\Ord\Mon\to \Ord}$ is monadic.
\end{proposition}
\begin{proof} We recall that, by Beck's monadicity criterion (see e.g. Th.2.4 in \cite{MS04}), a right adjoint functor $U_1$ is monadic if and only if
\begin{itemize}
\item $U_1$ reflects isomorphisms;

\item $\Ord\Mon$ has and $U_1$ preserves coequalizers of all parallel pairs $(f,g)$ such that $(U_1(f), U_1(g))$ has a contractible coequalizer in $\Ord$.
\end{itemize}

Given a morphism $f\colon (A, +, 0, \leq) \to (B, +, 0, \leq)$ in $\Ord\Mon$ such that $U_1(f)$ is an isomorphism in $\Ord$ then, being also a bijective homomorphism of monoids, it is an isomorphism of monoids and so it is also an isomorphism in $\Ord\Mon$.
Hence $U_1$ reflects isomorphisms.

For a parallel pair of morphisms $f, g \colon (A, +, 0, \leq) \to (B, +, 0, \leq)$ in $\Ord\Mon$ let $q\colon (B, +, 0) \to (C, +, 0)$ be a coequalizer of $(U_2(f), U_2(g))$ in the category of monoids. Considering in $C$ the preorder that is the transitive closure of the image by $q$ of the preorder in $B$, it is easy to prove that this  preorder is compatible with the monoid operation, so that $(C, +, 0, \leq) \in \Ord\Mon$, and also that
$$q \colon (B, +, 0, \leq) \to (C, +, 0, \leq)$$ is the coequalizer of $(f, g)$ in this category.

Let us assume that $(U_1(f), U_2(g))$ has a contractible coequalizer $ (U_1(f), U_1(g), h; i, j)$ in $\Ord$. We have to prove that the unique morphism $t \in \Ord$ such that $ t \cdot h = U_1( q)$ is an isomorphism.

Since $ V_2 U_1 = V_1 U_2$ and $V_1$ is monadic,  we know that $V_2(t)$ is a bijection. Furthermore, if $ c = t(x) \leq t(y) = d$ then $x \leq y$. Indeed, by definition of the preorder in $C$, there exists a zig-zag in $B$
\[ b_1 \leq b_2 \sim {b'}_2 \leq b_3 \cdots b_{n-1} \sim  {b'}_{n-1} \leq b_n, \]
such that $q(b_1) = c , q(b_n) = d$ and $q(b_i) = q ({b'}_i)$ for $i = 2, \cdots n-1$. Thus $ x = h(b_1) \leq  h(b_n) = y$.
\end{proof}

\begin{proposition}\label{porp: 3.9}
The subcategory $\Ord\Mon^{\square}$ is not coreflective in the category $\Ord\Mon$.
\end{proposition}
\begin{proof}
For every preordered set $(X, \leq)$, $F_1(X, \leq)= (X^{*},\text{con},[\;],\leq)$ has positive cone $P = \{[\;]\}$
that is a right normal (indeed a normal) submonoid. Hence the preordered monoid $F_1(X, \leq) \in \Ord\Mon^{\square}$ and we have the following situation
\[
\xymatrix{\Ord\Mon^{\square}\ar@<.9ex>[dr]^{{U_1}^{\square}}\ar[rr] && \Ord\Mon\simeq \Ord^{T} \ar@<-.9ex>[dl]_{U_1} \\ & \Ord  \ar@<-.9ex>[ur]_{F_1} \ar@<.9ex>[ul]^{{F_1}^{\square}} }
\]
where ${U_1}^{\square}$ is the restriction of $U_1$ to $\Ord\Mon^{\square}$, ${F_1}^{\square}$ is the corestriction of $F_1$ giving a left adjoint to ${U_1}^{\square}$,  and $T$ is the monad that both adjunctions induce in $\Ord$.

From that we conclude that $\Ord\Mon^{\square}$ cannot be coreflective in $\Ord\Mon$ otherwise, being closed under coequalizers, ${U_1}^{\square}$ would be monadic and so $\Ord\Mon^{\square} \cong \Ord^T \cong \Ord\Mon$
that is false as Example 2 shows.

\end{proof}

Direct proofs presented in this section are simple and informative about the categories involved.

However, since $\Ord\Mon$ is the category $\Mon(\Ord)$ of internal monoids in the category of preordered sets (which is not true for ordered groups) these results can be derived from more general ones relative to categories of models of the theory of monoids in monoidal categories. In our case, since $\Ord$ is a cartesian closed category which, furthermore, is  locally finitely presentable (see \cite{AR}), the construction of the left adjoint of $U_1\colon{\Ord\Mon=\Mon(\Ord)\to\Ord}$ is a particular case of the construction of the left adjoint of the forgetful functor of $\Mon(\C)\to\C$, when $\C$ is a symmetric monoidal category, satisfying some additional conditions, presented by G. M. Kelly in \cite{K80}, see also \cite{L10}. Also the monadicity of $U_1$ comes from Corollary 2.6 in \cite{P08}.

 In more detail, S. Lack proves in \cite{L10} that
the forgetful functor of $\Mon(\C)\to\C$ has a left adjoint when $\C$ is a symmetric monoidal category with countable coproducts that are preserved by tensoring on either side, with the free monoid over an object $X \in \C$  given by
\[1 + X + X^2 +  \cdots \]
where $X^n$ means the nth-tensoring of $X$.

In \cite{P08}, H. Porst deals with ``admissible monoidal categories" which are locally presentable categories that, in addition, are symmetric mo\-noi\-dal with the property  that tensoring by a fixed object defines a finitary functor (i.e., a functor preserving directed colimits).

In the cartesian case, that is when the tensor is given by the direct product and the identity is the terminal object
in the monoidal category, if $\C$ is locally presentable and cartesian closed it is clearly admissible, in the above sense, and so, by Corollary 2.6 in \cite{P08} we conclude the monadicity of $\Mon(\C)$ over $\C$.

\section{Schreier split extensions}

We recall that,  in the category of monoids, a Schreier split epimorphism (\cite{BMMS13}) is a diagram
\begin{equation}\label{eq: S-splitepi in Mon}
\xymatrix{
X\ar@<-.5ex>[r]_{k} & A \ar@<.5ex>[r]^{p}\ar@{-->}@<-.5ex>[l]_{q} & B \ar@<.5ex>[l]^{s}}
\end{equation}
where $k$, $p$ and $s$ are monoid homomorphisms, $p s = 1_B$, $k$ is the kernel of $p$ and $q$ is a set-theoretical map (called the Schreier retraction),
such that,
\begin{enumerate}
\item[(S1)] $ k q + s p = 1_ A$, and
\item[(S2)] $q (k(x) + s(b)) = x$, for  every $x \in X$ and $ b \in B$.
\end{enumerate}

\vspace*{1mm}
To the Schreier split epimorphism above corresponds an action
$$\varphi \colon B \to \End(X)$$
defined by $\varphi(b)(x) = q(s(b) + k (x))$ that we will denote by $b \cdot x$.

 Important consequences (\cite{BMMS13}), that will be used in the sequel, are the following:
\begin{enumerate}
\item[(C1)] $k(b\cdot x)+ s(b) = s(b) + k(x)$;
\item[(C2)] $q(a_1+a_2)=q(a_1)+q(sp(a_1)+kq(a_2))=q(a_1)+p(a_1)\cdot q(a_2)$, for all $a_1,a_2\in A$;
\item[(C3)] $A$ is isomorphic to the semi-direct product $X\rtimes_{\varphi}B$ with isomorphisms defined by $\alpha(a)=(q(a),p(a))$ and $\beta(x,b)=k(x)+s(b)$;
\item[(C4)] $p$ is the cokernel of $k$ and so, since the sequence is exact, we speak of Schreier split extensions.
\end{enumerate}

\vspace*{.5cm}

This definition can easily be extended to the category of preordered monoids by keeping $q$ a set-theoretical map and assuming that $k$, $p$ and $s$ are monotone homomorphisms.

In this section we are going to characterize  Schreier split extensions in $\Ord\Mon^{*}$. For that we use the isomorphism defined in Theorem \ref{thm: 2.11} and work in the category $\RNMono(\Mon)$. For simplicity, we assume that the objects in this category are inclusions and we denote the right normal submonoids of a monoid $M$ by $P_M$, since they are the positive cones of a compatible preorder in $M$.

\begin{definition}\label{def: Schreier split epis} A Schreier split epimorphism in $\RNMono(\Mon)$ is a diagram
\begin{equation}\label{eq: S-splitepi in RNMono}
\xymatrix{
P_X\ar[r]^{\bar{k}}\ar[d] & P_A \ar@<.5ex>[r]^{\bar{p}}\ar[d] & P_B \ar@<.5ex>[l]^{\bar{s}}\ar[d]\\
X\ar@<-.5ex>[r]_{k} & A \ar@<.5ex>[r]^{p}\ar@{-->}@<-.5ex>[l]_{q} & B \ar@<.5ex>[l]^{s}}
\end{equation}
in which the lower row is a Schreier split epimorphism in $\Mon$, and the upper row consists of right normal submonoids, the positive cones $P_X$, $P_A$, and  $P_B$, that make $X$, $A$, and $B$, objects in $\Ord\Mon^{*}$. The morphisms $\bar{k}$, $\bar{p}$, and $\bar{s}$, are the corresponding restrictions.
\end{definition}

We point out that we do not assume the monotonicity of $q$.

\vspace*{.5cm}

We will show that for every two objects $(X,P_X)$ and $(B,P_B)$ in $\RNMono(\Mon)$, there is an equivalence between Schreier split extensions of $(X,P_X)$ by $(B,P_B)$ and a certain kind of actions that we will call preordered actions for the purpose of this paper.

\begin{definition}\label{def: preordered action}
Let $(X,P_X)$ and $(B,P_B)$ be two objects in the category $\RNMono(\Mon)$. A \emph{preordered action} of $(B,P_B)$ on $(X,P_X)$, that will be denoted by $(X,B,P_X,P_B,\varphi,\xi)$, consists of a monoid action of the underlying monoids $B$ on $X$, i.e. a monoid homomorphism \[\varphi\colon{B\to\End(X)},\]
together with a set-theoretical mapping
\[\xi\colon{X\times P_B \to X},\]
satisfying the following conditions:
\begin{enumerate}
\item[(A1)] $\xi(0,b)=0$, for all $b\in P_B$
\item[(A2)] if $x\in P_X$ then $\xi(x,0)=x$
\item[(A3)] if $\xi(x,b)=x$ and $\xi(x',b')=x'$ then \[\xi(x+b\cdot x',b+b')=x+b\cdot x'\]
\item[(A4)] for all $x,u\in X$, $v\in P_B$, $b\in B$, if $\xi(u,v)=u$, then there exists $u'\in X$ such that \[x+b\cdot u=u'+v'\cdot x\] and \[\xi(u',v')=u'\]
where $v'\in P_B$ is such that $b+v=v'+b$, which exists because $P_B$ is right normal.
\end{enumerate}
\end{definition}

A morphism $(f_0,f_1,f_2)$ between two Schreier split extensions in the category $\RNMono(\Mon)$ is a commutative diagram of the form
\[\xymatrix{
&P_X\ar[dl]_{\bar{f_0}}\ar[rr]^{\bar{k}}\ar[dd] && P_A \ar[dl]_{\bar{f_1}}\ar@<.5ex>[rr]^{\bar{p}}\ar[dd] && P_B \ar[dl]_{\bar{f_2}}\ar@<.5ex>[ll]^{\bar{s}}\ar[dd]\\
P_{X'}\ar[rr]^(.7){\bar{k'}}\ar[dd] && P_{A'} \ar@<.5ex>[rr]^(.7){\bar{p'}}\ar[dd] && P_{B'} \ar@<.5ex>[ll]^(.3){\bar{s'}}\ar[dd]\\
& X\ar[dl]_{f_0}\ar@<-.5ex>[rr]_(.3){k} && A\ar[ld]_{f_1} \ar@<.5ex>[rr]^(.3){p}\ar@{-->}@<-.5ex>[ll]_(.7){q} && B \ar[ld]_{f_2} \ar@<.5ex>[ll]^(.7){s}\\
 X'\ar@<-.5ex>[rr]_{k'} && A' \ar@<.5ex>[rr]^{p'}\ar@{-->}@<-.5ex>[ll]_{q'} && B' \ar@<.5ex>[ll]^{s'}.}
 \]

Whereas a morphism of preordered actions,
$$(f_0,f_2)\colon{(X,B,P_X,P_B,\varphi,\xi) \to (X',B',P_X',P_B',\varphi',\xi')}$$ consists of two monoid homomorphisms $f_0\colon{X\to X'}$ and $f_2\colon{B\to B'}$ which  restrict to the respective positive cones giving  $\bar{f_0}\colon{P_X\to P_{X'}}$ and $\bar{f_2}\colon{P_B\to P_{B'}}$, such that $$f_0(b\cdot x)=f_2(b)\cdot f_0(x)$$ and \[\xi'(f_0(u),\bar{f_2}(v))=f_0(u),\]
whenever $\xi(u,v)=u$.
In other words, the diagram where the horizontal arrows are defined by the monoid actions, $(b,x) \mapsto b\cdot x$,
\[\xymatrix{B\times X\ar[d]_{f_0\times f_2}\ar[r]^{} & X\ar[d]^{f_0}\\B'\times X'\ar[r]^{} & X'}\]
is commutative and the diagram
\[\xymatrix{ X\times P_B\ar[d]_{f_0\times \bar{f_2}}\ar[r]^{\xi} & X\ar[d]^{f_0}\\ X'\times P_{B'}\ar[r]^{\xi'} & X'}\]
commutes only when restricted to those pairs $(u,v)\in X\times P_B$ for which $\xi(u,v)=u$. That is, there exists $g\colon{P_{\xi}\to P_{\xi'}}$, such that the left square and the outer rectangle commute
\begin{equation}\label{diag:g}
\xymatrix{P_{\xi}\ar[r] \ar@{-->}[d]_{g} & X\times P_B\ar[d]_{f_0\times \bar{f_2}}\ar[r]^{\xi} & X\ar[d]^{f_0}\\P_{\xi'}\ar[r] & X'\times P_{B'}\ar[r]^{\xi'} & X'}
\end{equation}
where $P_{\xi}=\{(u,v)\in X\times P_B\mid \xi(u,v)=u\}$ and similarly for $P_{\xi'}$.

This way we defined a category $\S$ of Schreier split extensions in $\RNMono(\Mon)$ and a category $\A$ of preordered actions.

\begin{theorem}
There is an equivalence of categories between the category $\A$ of preordered actions and the category $\S$ of Schreier split extensions in $\RNMono(\Mon)$.
\end{theorem}
\begin{proof}
We define a functor $G\colon{\S\to\A}$ assigning to a Schreier split epimorphism in $\RNMono(\Mon)$ as displayed in $(\ref{eq: S-splitepi in RNMono})$, a preordered action as follows:
\begin{enumerate}
\item $\varphi_b(x)=q(s(b)+k(x))$, for all $x\in X$ and $b\in B$;
\item $\xi(u,v)=u$ if $k(u)+s(v)\in P_A$ and $\xi(u,v)=0$ otherwise.
\end{enumerate}

 These maps $\varphi$ and $\xi$ satisfy the conditions of Definition \ref{def: preordered action}:
\begin{enumerate}
\item[$\bullet$] the first condition above defines an action of $B$ on $X$ (\cite{MMS13}).

\item[$\bullet$] $\xi(0,b)=0$

\item[$\bullet$] $\xi(x,0)=x$ since $k(x)+s(0)=k(x)\in P_A$
\item[$\bullet$] If $\xi(x,b)=x$ and $\xi(x',b')=x'$ then $k(x)+s(b),k(x')+s(b')\in P_A$. Since $P_A$ is a monoid then $$k(x)+s(b)+k(x')+s(b')\in P_A,$$ but $s(b)+k(x')=k(b\cdot x')+s(b)$ and so we have that  $$k(x+b\cdot x')+s(b+b')\in P_A.$$ Consequently, $\xi(x+b\cdot x',b+b')=x+b\cdot x'$.

\item[$\bullet$] $P_A\to A\cong X\rtimes_{\varphi} B$ right normal means that for all $(x,b)\in X\rtimes_{\varphi} B$, $(u,v)\in P_A$,  there exists $(u',v')\in P_A$ such that \[(x,b)+(u,v)=(u',v')+(x,b)\]
that is
\[(x+b\cdot u,b+v)=(u'+v'\cdot x, v'+b)\]
which implies $x+b\cdot u=u'+v'\cdot x$ and $b+v=v'+b$.
\end{enumerate}

Defining   $G(f_0,f_1,f_2)=(f_0,f_2)$ we obtain a functor $G\colon{\S\to\A}$.

\vspace*{0.1cm}
Conversely, given a preordered action $(X,B,P_X,P_B,\varphi,\xi)$ we construct a Schreier split extension in $\RNMono(\Mon)$ as follows (using the same notation as in $(\ref{eq: S-splitepi in RNMono})$):
\begin{enumerate}
\item $A=X\rtimes_{\varphi} B$ is the semi-direct product of the underlying monoids induced by the monoid action $\varphi$.
 This means that $A$ is the set $X\times B$ with the monoid operation
\[(x,b)+(x',b')=(x+b\cdot x',b+b')\] and neutral element $(0,0)\in X\times B$;
\item the right normal submonoid of A, $P_A=P_{\xi}$, is defined by
\[(x,b)\in P_A\Leftrightarrow b\in P_B \text{ and } \xi(x,b)=x.\]
This gives a Schreier split extension in $\RNMono(\Mon)$. Indeed:
\begin{enumerate}
\item $P_{\xi}$ is a submonoid of $X\rtimes_{\varphi}B$ by (A3) and the fact that $P_B$ is a monoid.
\item The right normality of $P_A$ comes from (A4).
\item The morphism $\langle 1,0\rangle\colon{X\to A}$ restricts to $P_X\to P_A$ by (A2).
\item The morphism $\langle 0,1\rangle\colon{B\to A}$ restricts to $P_B\to P_A$ by (A1).
\end{enumerate}
\end{enumerate}

Moreover, we define a functor $H\colon{\A\to \S}$ assigning to each morphism of actions $$(f_0,f_2)\colon{(X,B,P_,P_B,\varphi,\xi)\to(X',B',P_{X'},P_{B'},\varphi',\xi')},$$
$H(f_0,f_2)=(f_0,f_1,f_2)$ where $f_1=g\colon{P_{\xi}\to P_{\xi'}}$  as in diagram $(\ref{diag:g})$.

\vspace{.5cm}

Then  $GH\cong 1_{\A}$: in the diagram
\[\xymatrix{
&P_X\ar[dr]_{}\ar[rr]^{\bar{k}}\ar[dd] && P_A \ar@<.5ex>[rr]^{\bar{p}}\ar[dd] && P_B \ar@{=}[dl]_{}\ar@<.5ex>[ll]^{\bar{s}}\ar[dd]\\
 && P_{\xi}\ar[ur]_{\bar{\beta}} \ar@<.5ex>[rr]^(.7){}\ar[dd] && P_{B} \ar@<.5ex>[ll]^(.3){}\ar[dd]\\
& X\ar[dr]_{}\ar@<-.5ex>[rr]_(.3){k} && A \ar@<.5ex>[rr]^(.3){p}\ar@{-->}@<-.5ex>[ll]_(.7){q} && B \ar@{=}[ld]_{} \ar@<.5ex>[ll]^(.7){s}\\
 && X\rtimes_{\varphi}B\ar[ru]_{\beta} \ar@<.5ex>[rr]^{} && B \ar@<.5ex>[ll]^{}}
 \]
since $\beta(x,b)=k(x)+s(b)$, by definition of $P_{\xi}$, we conclude that $\bar{\beta}\colon{P_{\xi}\to P_A}$ is an isomorphism.

 It is easy to check that also $HG\cong 1_{\S}$, thus giving the desired equivalence of categories.

\end{proof}

Finally, we point out two interesting particular cases:
\begin{enumerate}
\item[$\bullet$] When $q$ is a monotone map then it restricts to $\bar{q}\colon{P_A\to P_X}$ and $\xi$ is trivial, in the sense that $\xi(x,b)=x$ when $x\in P_X$ and $b\in P_B$ and it is zero otherwise. In this case, the upper row of the diagram (\ref{eq: S-splitepi in RNMono}) is a Schreier split epimorphism of monoids and hence $P_A$ is isomorphic to the semidirect product $P_X\times_{\bar{\varphi}} P_B$.

\item[$\bullet$] When $q$ is an homomorphism then the monoid action  $\varphi$ is trivial, i.e. $\varphi_b(x)=x$, for all $b\in B$. However, we may still have a non trivial $\xi$ in this case, as the following example shows.
\end{enumerate}

In the diagram $(\ref{eq: S-splitepi in RNMono})$ if $q$ is a monoid homomorphism then $A\cong X\times B$ but the upper row need not be a Schreier split epimorphism.

\begin{example}

Let us consider the following diagram
\begin{equation}\label{eq: example}
\xymatrix{
\{0\}\ar[r]^{}\ar[d] & \N\times \N \ar@<.5ex>[r]^{+}\ar[d] & \N \ar@<.5ex>[l]^{\langle 0,1\rangle}\ar[d]\\
\mathbb{Z} \ar@<-.5ex>[r]_{\langle 1,-1\rangle} & \mathbb{Z}\times \mathbb{Z} \ar@<.5ex>[r]^{+}\ar@{-->}@<-.5ex>[l]_{\pi_1} & \mathbb{Z}, \ar@<.5ex>[l]^{\langle 0,1\rangle}}
\end{equation}
which is an example of a Schreier split epimorphism in the category $\RNMono(\Mon)$. The left $\mathbb{Z}$ has the discrete order because its positive positive cone is $\{0\}$, while the one on the right has the usual order since its positive cone is $\N$. The positive cone $\N\times\N$ and the corresponding order in $\mathbb{Z}\times \mathbb{Z}$ will be described below.

In this case we have a non trivial $\xi\colon{\mathbb{Z}\times \N \to \mathbb{Z}}$, defined by
\[
 \xi(u,v)= \left\{
 \begin{array}{rl}
 u & \mbox{if $u\in \N$ and $u\leq v$}\\
 0 & \mbox{otherwise}
 \end{array}
   \right.
 \]
giving a preordered action $(\mathbb{Z},\mathbb{Z},\{0\},\N,\varphi,\xi)$ where $\varphi$ is trivial,  which induces a Schreier split extension in $\RNMono(\Mon)$
\begin{equation}\label{eq: example_a}
\xymatrix{
\{0\}\ar[r]^{}\ar[d] & P_{\xi} \ar@<.5ex>[r]^{}\ar[d] & \N \ar@<.5ex>[l]^{}\ar[d]\\
\mathbb{Z} \ar@<-.5ex>[r]_{\langle 1,0\rangle} & \mathbb{Z}\times \mathbb{Z} \ar@<.5ex>[r]^{\pi_2}\ar@{-->}@<-.5ex>[l]_{\pi_1} & \mathbb{Z} \ar@<.5ex>[l]^{\langle 0,1\rangle}}
\end{equation}
where $P_{\xi}=\{(u,v)\in \mathbb{Z}\times \mathbb{Z}\mid 0\leq u\leq v\}$, with $0\leq u\leq v$ in the usual order of $\N$. This defines the positive cone $P = P_ {\xi}$ and the order of $\mathbb{Z}\times \mathbb{Z}$ in $(\ref{eq: example})$.
\end{example}










\begin{thebibliography}{99}

\bibitem{AR} J. Ad\'{a}mek and J. Rosick\'y,  \emph{Locally presentable and accessible categories}, Cambridge Univ. Press 1994.

\bibitem{BMMS13} D. Bourn, N. Martins-Ferreira, A. Montoli and M. Sobral, \emph{Schreier split epimorphisms in monoids and in semirings}, Textos de Matem\'{a}tica S\'{e}rie B \textbf{45} (Departamento de Matem\'{a}tica, Universidade de Coimbra (ISBN 978-972-8564-49-0)) (2013) vi+116pp.

\bibitem{BMMS14} D. Bourn, N. Martins-Ferreira, A. Montoli, M. Sobral, \emph{Schreier split epimorphisms between monoids}, Semigroup Forum \textbf{88} (2014) 739--752.

\bibitem{CMM19} M. M. Clementino, N. Martins-Ferreira and A. Montoli, \emph{On the categorical behaviour of preordered groups}, J. Pure Appl. Algebra \textbf{223} (2019) 4226--4245.

\bibitem{GMRS19} A. P. Garr\~{a}o, N. Martins-Ferreira, M. Raposo, M. Sobral, \emph{Cancellative conjugation semigroups and monoids}, Semigroup Forum \textbf{} (doi:10.1007/s00233-019-10070-9) (2019) 1--31.

\bibitem{Green} J. A. Green,  \emph{On the structure of semigroups}, Annals of Mathematics (second series) \textbf{54}(1) (1951) 163--172.

\bibitem{G74} P. A. Grillet, \emph{Left coset extensions}, Semigroup Forum \textbf{7} (1974) 200--263.

\bibitem{H98} P. H\'{a}jek, \emph{Metamathematics of Fuzzy Logic}, Kluwer Academic Publishors, Dordrecht, 1988.

\bibitem{HP99} K. Henckell and J. -E. Pin, \emph{Ordered monoids and J-trivial monoids},

\bibitem{JV19} J. Janda and T. Vetterlein, \emph{The coextension of commutative pomonoids and its application to triangular norms
}, Quaestiones Mathematicae \textbf{42}(3) (2019) 319--345.

\bibitem{K80} G. M. Kelly, \emph{A unified treatment of transfinite constructions for free algebras, free monoids, colimits, associated sheaves, and so on}, Bulletin of the Australian Mathematical Society \textbf{22}(1) (1980) 1--83.

\bibitem{L10} S. Lack, \emph{Note on the construction of free monoids}, Applied Categorical Structures \textbf{18} (2010) 17--29.

\bibitem{MacLane} S. Mac Lane, \emph{Categories for the Working Mathematician}, 2ed, Graduate Texts in Mathematics 5, Springer, 1998.

\bibitem{L82} J. Leech, \emph{Extending groups by monoids}, Journal of Algebra \textbf{74} (1982) 1--19.

\bibitem{L75} J. Leech, \emph{H-coextensions of monoids}, Mem. Amer. Math. Soc.  \textbf{157} (1975) 1--66.

\bibitem{MS04}  J. MacDonald and M. Sobral, \emph{Aspects of Monads} In M. Pedicchio \& W. Tholen (Eds.), \emph{Categorical Foundations: Special Topics in Order, Topology, Algebra, and Sheaf Theory} (Encyclopedia of Mathematics and its Applications: Cambridge University Press
pp. 213--268) 2004.



\bibitem{MMS13} N. Martins-Ferreira, A. Montoli and M. Sobral, \emph{Semidirect products and crossed modules in monoids with operations}, J. Pure Appl. Algebra  \textbf{217} (2) (2013)  334--347 .

\bibitem{P08} H. -E. Porst, \emph{On categories of monoids, comonoids and bimodules},
Quaestiones Mathematicae \textbf{31}(2) (2008) 127--139.

\bibitem{P98} A. Patchkoria, \emph{Crossed semimodules and Schreier internal categories in the category of monoids}, Georgian Mathematical Journal \textbf{5}(6) (1998) 575--581.

\bibitem{Pin95} J. -E. Pin, \emph{ A variety theorem without complementation}, Isevestiya VUZ Mathematika \textbf{39}(1995) 80--90. English version. Russian Mathem. Iz. VUZ \textbf{39} (1995) 74--83.


\bibitem{PW97} J. -E. Pin and P. Weil, \emph{Polynomial closure and unambiguous product}, Thery Comput. Systems \textbf{30}(1997) 1--39.

\bibitem{PW02} J. -E. Pin and P. Weil, \emph{Semidirect product of ordered monoids}, Commun.Algebra \textbf{30}(1)(2002) 149--169.

 \bibitem{RS97} G. Rozenberg and A. Salomaa (Ed.), \emph{Handbook of formal languages}, Vol \textbf{1} Springer - Verlag 1997.

\bibitem{S75} I. Simon, \emph{Piecewise testable events}, Proc. 2nd  GI Conf., Lect. Notes in Comp. Sci. \textbf{33}, Springer Verlag, Berlin, Heidelberg, New York (1975 ), 214--222.

\bibitem{ST85} H. Straubing and D. Th\'erien, \emph{Partially ordered finite monoids and a theorem of I. Simon}, J. of Algebra \textbf{119}, (1985), 393--399.

\bibitem{FW92} F. Wehrung, \emph{Injective positively ordered monoids I}, Journal of Pure and Appl. Algebra \textbf{83}(1)(1992) 43--82.






\end{thebibliography}

\end{document}